\documentclass[11pt,reqno]{amsart}
\usepackage{amscd,amsmath,amsopn,amssymb,amsthm,multicol}
\usepackage{amscd,amsmath,amsopn,amssymb,amsthm,multicol}
\usepackage[backref=page,breaklinks=true,colorlinks=true,linkcolor=blue,citecolor=blue,urlcolor=blue]{hyperref}

\DeclareMathOperator{\ad}{ad}
\DeclareMathOperator{\Id}{Id}

\DeclareMathOperator{\const}{const}
\DeclareMathOperator{\diag}{diag}

\DeclareMathOperator{\Ad}{Ad}

\DeclareMathOperator{\Aut}{Aut}

\DeclareMathOperator{\Ker}{Ker}
\DeclareMathOperator{\Imm}{Im}
\DeclareMathOperator{\Lie}{Lie}

\renewenvironment{proof}[1][Proof]{\textbf{#1.} }
{\ \rule{0.5em}{0.5em}}
\newtheorem{theorem}{Theorem}
\newtheorem{prop}{Proposition}
\newtheorem{lemma}{Lemma}
\newtheorem{corollary}{Corollary}

\newtheorem{conjecture}{Conjecture}

\theoremstyle{definition}

\newtheorem{remark}{Remark}

\begin{document}

\title
[Spectral properties of Killing vector fields of constant length] {Spectral properties of Killing vector fields \\ of constant length}

\author{Yu.G.~Nikonorov}
\address{Yu.\,G. Nikonorov \newline
Southern Mathematical Institute of Vladikavkaz Scientific Centre \newline
of the Russian Academy of Sciences, Vladikavkaz, Markus st. 22, \newline
362027, Russia}
\email{nikonorov2006@mail.ru}

\begin{abstract}
This paper is devoted to the study of properties of Killing vector fields of constant length on Riemannian manifolds.
If $\mathfrak{g}$ is a Lie algebra of Killing vector fields on a given Riemannian manifold $(M,g)$,
and $X\in \mathfrak{g}$ has constant length on $(M,g)$, then we prove that the linear operator $\ad(X):\mathfrak{g} \rightarrow \mathfrak{g}$
has a pure imaginary spectrum. More detailed structure results on the corresponding operator $\ad(X)$ are obtained.
Some special examples of vector fields of constant length are constructed.

\vspace{2mm} \noindent 2010 Mathematical Subject Classification:
53C20, 53C25, 53C30.

\vspace{2mm} \noindent Key words and phrases: geodesic orbit space, homogeneous Riemannian space, Killing vector field of constant length.
\end{abstract}

\maketitle

\section{Introduction}

We study general properties of a given Killing vector field of constant length  $X$ on an arbitrary Riemannian manifold $(M,g)$.
A comprehensive survey on classical results in this direction could be found in \cite{BerNik2008, BerNikClif}.
Important properties of Killing vector fields of constant length (abbreviated as KVFCL) on compact homogeneous Riemannian spaces are studied in \cite{Nik2015}.
Some resent results about Killing vector field of constant length on some special Riemannian manifolds  are obtained in
\cite{WolfPodXu2017, XuWolf2016}. All manifolds in this paper are supposed to be connected.

Let us consider a Riemannian manifold $(M,g)$ and
any Lie group $G$ acting effectively on $(M,g)$ by isometries.
We will identify the Lie algebra $\mathfrak{g}$ of $G$ with the corresponded Lie algebra of Killing vector field on $(M,g)$ as follows.
For any $U\in \mathfrak{g}$ we consider a one-parameter group $\exp(tU)\subset G$ of isometries of $(M,g)$ and define a Killing vector field $\widetilde{U}$
by a usual formula
\begin{equation}\label{kfaction}
\widetilde{U} (x) = \left.\frac{d}{dt}\exp(tU)(x)\right|_{t=0} .
\end{equation}
It is clear that the map $U \rightarrow \widetilde{U}$ is linear and injective, but $[\widetilde{U},\widetilde{V}]= -\widetilde{[U,V]_{\mathfrak{g}}}$,
where $[\cdot,\cdot]_{\mathfrak{g}}$ is the Lie bracket in $\mathfrak{g}$ and $[\cdot,\cdot]$ is the Lie bracket of vector fields on $M$.
We will use this identification repeatedly in this paper.

Any $X\in \mathfrak{g}$ determines a linear operator $\ad(X):\mathfrak{g} \rightarrow \mathfrak{g}$ acting by $Y \mapsto [X,Y]$.
If we consider $X$ as a Killing vector field on $(M,g)$, then some geometric type assumptions on $X$ imply special properties (in particular, spectral properties)
of the corresponding operator $\ad(X)$. In this paper, we study the property of $X$ to be of constant length.

For a Lie algebra $\mathfrak{g}$, we denote by $\mathfrak{n(g)}$ and $\mathfrak{r(g)}$ the nilradical (the maximal nilpotent ideal)
and the radical of $\mathfrak{g}$ respectively.
A maximal semi-simple subalgebra of $\mathfrak{g}$ is called a Levi factor or a~Levi subalgebra.
There is a semidirect decomposition $\mathfrak{g}=\mathfrak{r(g)}\rtimes \mathfrak{s}$, where $\mathfrak{s}$ is an arbitrary Levi factor.
The Malcev--Harish-Chandra theorem
states that any two Levi factors of $\mathfrak{g}$ are conjugate by an automorphism $\exp(\Ad (Z))$
of $\mathfrak{g}$, where $Z$ is in the nilradical $\mathfrak{n(g)}$ of~$\mathfrak{g}$.
We have $\mathfrak{r(g)}=[\mathfrak{s},\mathfrak{r(g)}] \oplus C_{\mathfrak{r(g)}}(\mathfrak{s})$ (a direct sum of linear subspaces),
where $C_{\mathfrak{r(g)}}(\mathfrak{s})$ is the centralizer of
$\mathfrak{s}$ in $\mathfrak{r(g)}$.
Recall also that $[\mathfrak{g},\mathfrak{r(g)}]\subset \mathfrak{n(g)}$, therefore,
$[\mathfrak{s},\mathfrak{r(g)}]\subset [\mathfrak{g},\mathfrak{r(g)}]\subset \mathfrak{n(g)}$. Moreover,
$D(\mathfrak{r(g)})\subset \mathfrak{n(g)}$ for every derivation
$D$ of $\mathfrak{g}$. For any Levi factor $\mathfrak{s}$, we have
$[\mathfrak{g},\mathfrak{g}]=[\mathfrak{r(g)}+\mathfrak{s},\mathfrak{r(g)}+\mathfrak{s}]=[\mathfrak{g},\mathfrak{r(g)}]\rtimes \mathfrak{s}
\subset \mathfrak{n(g)}\rtimes \mathfrak{s}$.
For a more detailed discussion of the Lie algebra structure we refer to \cite{HilNeeb}.

\smallskip

Recall that a subalgebra $\mathfrak{k}$ of a Lie algebra $\mathfrak{g}$ is said to
be {\it compactly embedded in} $\mathfrak{g}$ if $\mathfrak{g}$ admits an inner product relative to which the
operators $\ad(X):\mathfrak{g} \mapsto \mathfrak{g}$, $X\in \mathfrak{k}$, are skew-symmetric.
This condition is equivalent to the following condition:
the closure of $\Ad_G(\exp(\mathfrak{k}))$ in $\Aut(\mathfrak{g})$ is compact,
see e.~g. \cite{HilNeeb}.
Note that for a compactly embedded subalgebra $\mathfrak{k}$, every operator $\ad(X):\mathfrak{g}\rightarrow \mathfrak{g}$,
$X\in \mathfrak{k}$, is semisimple and the spectrum  of  $\ad(X)$ lies in ${\bf i}\mathbb{R}$, where ${\bf i}=\sqrt{-1}$.
Recall also that a subalgebra $\mathfrak{k}$ of a Lie algebra $\mathfrak{g}$ is said to
be {\it compact} if it is compactly embedded in itself.
It is equivalent to the fact that there is a compact Lie group with a given Lie algebra $\mathfrak{k}$.
It is clear that any compactly embedded subalgebra $\mathfrak{k}$ of  $\mathfrak{g}$ is compact.
\medskip

One of the main results of this paper is the following

\begin{theorem}\label{conjec1}
For any Killing field of constant length $X\in \mathfrak{g}$  on a Riemannian manifold $(M,g)$,
the spectrum of the operator $\ad(X):\mathfrak{g}\rightarrow \mathfrak{g}$ is pure imaginary, i.~e.
is in ${\bf i}\mathbb{R}$.
\end{theorem}

This result is a partial case of Theorem \ref{the.eigen1}.
It is clear that it is non-trivial only for noncompact Lie algebras $\mathfrak{g}$ and only when $X$ is not a central element of $\mathfrak{g}$.
On the other hand, there are examples of Killing fields of constant length $X\in \mathfrak{g}$ for noncompact $\mathfrak{g}$.
Moreover, Theorems \ref{ideal2} and \ref{on_nilrad1} give us examples when $X\in \mathfrak{n(g)}$ and $\ad(X)$ is non-trivial and nilpotent.
In particular, $\ad(X)$ is not semisimple in this case. This observation leads to the following conjecture.

\begin{conjecture}\label{conjec2}
If $\mathfrak{g}$ is semisimple, then any Killing field of constant length $X\in \mathfrak{g}$  on  $(M,g)$ is a compact vector in $\mathfrak{g}$, i.~e.
the Lie
algebra $\mathbb{R}\cdot X$ is compactly embedded in $\mathfrak{g}$.
\end{conjecture}

The paper is organized as follows. In Section 2,
we establish some important spectral properties of the operator
$\ad(X):\mathfrak{g} \mapsto \mathfrak{g}$ for any Killing vector field of constant length $X\in  \mathfrak{g}$ on a given Riemannian manifold $(M,g)$.
One of the main results is Theorem \ref{decomp1}, that implies non-trivial geometric properties of $(M,g)$ in the case when
the Lie algebra $\mathfrak{g}$ could be decomposed as a direct Lie algebra sum.
In Section 3, we obtain some results on  Killing vector field of constant length on geodesic orbit Riemannian spaces.
In particular, Theorems \ref{ideal2} and \ref{on_nilrad1} imply that any Killing vector field in the center of $\mathfrak{n}(\mathfrak{g})$
has constant length on a given geodesic orbit space $(G/H, g)$. This observation provides
non-trivial examples of Killing vector field of constant length $X$ such that the operator $\ad(X):\mathfrak{g} \mapsto \mathfrak{g}$ is nilpotent.
\medskip

\section{KVFCL on general Riemannian manifolds}

In  what follows, we assume that a Lie group $G$ acts effectively on a Riemannian manifold $(M,g)$ by isometries,
$\mathfrak{g}$ is the Lie algebra of $G$, elements of $\mathfrak{g}$ are identified with Killing vector field on $(M,g)$ according to
(\ref{kfaction}).

The following characterizations of Killing vector fields of constant length on Riemannian manifolds is very useful.

\begin{lemma}[Lemma 3 in \cite{BerNikClif}]\label{lemma1}
Let $X$ be a non-trivial Killing vector field on a Riemannian manifold
$(M,g)$. Then the following conditions are equivalent:

{\rm 1)} $X$ has constant length on $M$;

{\rm 2)} $\nabla_XX=0$ on $M$;

{\rm 3)} every integral curve of the field $X$ is a geodesic in $(M,g)$.
\end{lemma}

\begin{lemma}[Lemma 2 in \cite{Nik2016}]\label{le1}
If a Killing vector field $X\in \mathfrak{g}$ has constant length on $(M,g)$, then for any
$Y,Z\in \mathfrak{g}$ the  equalities
\begin{eqnarray}
g([Y,X],X)=0 \,,\label{e0}\\
g([Z,[Y,X]],X)+g([Y,X],[Z,X])=0\,\,\, \label{e00}
\end{eqnarray}
hold at every point of $M$.
If $G$ acts on $(M,g)$ transitively, then condition (\ref{e0}) implies that $X$ has constant length.
Moreover, the condition (\ref{e00}) also implies that $X$ has constant length for compact $M$ and transitive $G$.
\end{lemma}

Now, we are going to get some more detailed results.

\begin{prop}\label{prop.1n}
Let $X \in \mathfrak{g}$ be a Killing vector field of constant length on $(M,g)$.
Denote by $C_{\mathfrak{g}}(X)$ the centralizer of $X$ in $\mathfrak{g}$ and consider $P(X):=\{Y\in \mathfrak{g}\,|\, g(X,Y)=0 \,\mbox{ on }\, M\}$.
Then we have $[X,\mathfrak{g}]\subset P(X)$ and $[Z, P(X)]\subset P(X)$ for any $Z\in C_{\mathfrak{g}}(X)$.
\end{prop}

\begin{proof}
By Lemma \ref{le1} we have $g([Y,X],X)=0$ for any $Y \in \mathfrak{g}$, hence, $[X,\mathfrak{g}]\subset P(X)$.
If $Y\in P(X)$, then $g(X,Y)=0$. Therefore, for any $Z\in C_{\mathfrak{g}}(X)$ we get
$g(X,[Z,Y])=Z\cdot g(X,Y)=0$ on $M$, i.~e. $[Z, P(X)]\subset P(X)$.
\end{proof}

\begin{theorem}\label{decomp1}
Let $X \in \mathfrak{g}$ be a Killing vector field of constant length on $(M,g)$.
Suppose that we have a direct Lie algebra sum $\mathfrak{g}=\mathfrak{g}_1\oplus \cdots \oplus \mathfrak{g}_l$, $l \geq 2$.
Then for every $i=1,\dots,l$, there is an ideal $\mathfrak{u}_i$ in $\mathfrak{g}_i$ such that
$[X,\mathfrak{g}_i] \subset \mathfrak{u}_i$ and $g(\mathfrak{u}_i,\mathfrak{u}_j)=0$  on $M$ for every $i\neq j$.
\end{theorem}

\begin{proof} Since $X$ is of constant length, then $\mathfrak{g}_i \cdot g(X,X)=g([\mathfrak{g}_i,X],X)=0$ for any $i$ by Lemma~\ref{le1}.
If we take $j \neq i$, then
$$
0=\mathfrak{g}_j \cdot g([\mathfrak{g}_i,X],X)=g([\mathfrak{g}_i,X],[\mathfrak{g}_j,X]).
$$
Let $\{\mathfrak{u}_i\}$, $i=1,\dots,l$, be a set of maximal (by inclusion) subspaces $\mathfrak{u}_i\subset \mathfrak{g}_i$,
such that $[\mathfrak{g}_i,X]\subset \mathfrak{u}_i$ and
$g(\mathfrak{u}_i,\mathfrak{u}_j)=0$ for every $i \neq j$ (such a set of subspaces should not be unique in general). Since
$$
0=\mathfrak{g}_i \cdot g(\mathfrak{u}_i,\mathfrak{u}_j)=0=g([\mathfrak{g}_i,\mathfrak{u}_i],\mathfrak{u}_j),
$$
then $[\mathfrak{g}_i,\mathfrak{u}_i] \subset \mathfrak{u}_i$ due to the choice of $\mathfrak{u}_i$. Hence, every $\mathfrak{u}_i$ is an ideal in
$\mathfrak{g}_i$ and in $\mathfrak{g}$.
\end{proof}

\begin{remark}
If $X=X_1+X_2+\cdots +X_l$, where $X_i \in \mathfrak{g}_i$, then $\mathfrak{u}_i\neq 0$ if $X_i$ is not in the center of $\mathfrak{g}_i$.
In particular, if $X_i\neq 0$ and $\mathfrak{g}_i$ is simple, then $\mathfrak{u}_i=\mathfrak{g}_i$.
Note, that Theorem~\ref{decomp1} leads to a more simple proof of Theorem 1 in \cite{Nik2015} about
properties of Killing vector fields of constant length on compact
homogeneous Riemannian manifolds. See also Remark \ref{dirdecgo} about geodesic orbit spaces.
\end{remark}

\begin{prop}\label{decomp2}
Let $X \in \mathfrak{g}$ be a Killing vector field of constant length on $(M,g)$.
Then for every $V,W \in [X,\mathfrak{g}]$ we have the equality
\begin{equation}\label{eq.nilport1}
g([X,V],W)+g(V,[X,W])=0
\end{equation}
on $M$.
\end{prop}

\begin{proof} Taking in mind the polarization, it suffices to prove $g([X,V],V)=0$ on $M$ for every $V \in [X,\mathfrak{g}]$.
Take any $U$ such that $[X,U]=V$.
Since $X$ has constant length, we have $g(X,[U,X])=0$ according to (\ref{e0}). Hence,
$$
g([V,X],V)=g\bigl([[X,U],X], [X,U]\bigr)=[X,U]\cdot g(X,[X,U])=0,
$$
that proves the proposition.
\end{proof}

\smallskip

{\it In what follows, for  a Killing vector field of constant length $X \in \mathfrak{g}$ on $(M,g)$,
we denote by $L=L(X)$ the linear operator  $\ad(X):\mathfrak{g} \rightarrow \mathfrak{g}$.}

\smallskip

\begin{prop}\label{the.eigen0}
Let $X \in \mathfrak{g}$ be a Killing vector field of constant length on $(M,g)$. Then

1) $L=\ad(X)$ has no non-zero real eigenvalue;

2) If $[[X,Z],Z]\in [X,\mathfrak{g}]$ {\rm(}in particular, if $[[X,Z],Z]= 0${\rm)}
for some $Z\in \mathfrak{g}$, then we get $[X,Z]=0$;

3) If $\mathfrak{a}$ is an $\ad(X)$-invariant subspace in $\mathfrak{g}$ such that $[\mathfrak{a},\mathfrak{a}]\subset [X,\mathfrak{g}]$,
then $[X,\mathfrak{a}]=0$;

4) If $\mathfrak{a}$ is any abelian ideal in $\mathfrak{g}$ {\rm(}one may take $\mathfrak{a}=C(\mathfrak{n}(\mathfrak{g}))$ in particular{\rm)},
then $[X,\mathfrak{a}]=0$.
\end{prop}

\begin{proof} 1) Suppose the contrary, i.e. there is nontrivial $Y\in \mathfrak{g}$ such that $[X,Y]=\lambda Y$ for some real $\lambda \neq 0$.
Then $[Y,[Y,X]]=-\lambda[Y, Y]=0$ and we get
$$
0=g([Y,[Y,X]],X)+g([Y,X],[Y,X])=g([Y,X],[Y,X])=\lambda^2 g(Y,Y)
$$
by Lemma \ref{le1}, that is impossible.

2) If $[[X,Z],Z]= [X,U]$ for some $U\in \mathfrak{g}$, then
$g([[X,Z],Z],X)=g([X,U],X)=0$. Therefore,
$g([Z,X],[Z,X])=g([Z,[Z,X]],X)+g([Z,X],[Z,X])=0$ by Lemma \ref{le1}. Hence, $[X,Z]=0$.

3) Take any $Z\in \mathfrak{a}$. Since $[X,Z]\in \mathfrak{a}$, we get $[[X,Z],Z]\in [X,\mathfrak{g}]$. Then we have
$[X,Z]=0$ by 2).

4) This is a partial case of 3).
\end{proof}

\begin{theorem}\label{the.eigen1}
Let $X \in \mathfrak{g}$ be a Killing vector field of constant length on $(M,g)$. Then the following assertions hold.

1) We have an $L$-invariant linear space decomposition $\mathfrak{g}=A_1\oplus A_2$, where $A_1=\Ker(L^2)$ and $A_2=\Imm(L^2)$.
Moreover, $A_1$ is the root space for $L$ with the eigenvalue $0$ and $L$ is invertible on~$A_2$.

2) If $\mathfrak{o}$ is a 2-dimensional $L$-invariant subspace, corresponding to a complex conjugate pair of eigenvalues
$\alpha\pm  \beta{\bf i\,}$ ($\beta\neq 0$), i.~e. $L(U)=\alpha \cdot U-\beta \cdot V$ and $L(V)=\beta \cdot U+\alpha \cdot V$ for some non-trivial
$U,V \in \mathfrak{o}$, then
$\alpha=0$, $g(U,V)=0$, and $g(U,U)=g(V,V)$ on $G/H$.

3) All eigenvalues of $L$ have trivial real parts.
\end{theorem}

\begin{proof} 1) Let us fix an arbitrary $Z\in \mathfrak{g}$. If we put $V=[X,Z]$ and $W=[X,[X,Z]]$ in~(\ref{eq.nilport1}), we get
\begin{eqnarray*}
g([X,[X,Z]],[X,[X,Z]])+g([X,Z],[X,[X,[X,Z]]])\\
=g(L^2(Z),L^2(Z))+g(L(Z),L^3(Z))=0.
\end{eqnarray*}
From this we see that $L^3(Z)=0$ implies $L^2(Z)=0$. This observation implies $L^3(\mathfrak{g})=L^2(\mathfrak{g})=\Imm(L^2)$.
In particular, we get that $L$ is invertible on $A_2=\Imm(L^2)$.

Let us prove that $A_1\cap A_2=\Ker(L^2) \cap \Imm(L^2)=0$. Suppose that there is a non-trivial $W\in \Ker(L^2) \cap \Imm(L^2)$. Then
$L^2(W)=0$ and there is $V\in \mathfrak{g}$ such that $W=L^2(V)$. If $L(W)=0$, then we have $L^2(V)\neq 0$ and $L^3(V)=0$, that is impossible.
If $L(W)\neq 0$, then we have $L^2(L(V))=L(W)\neq 0$ and $L^3(L(V))=L^2(W)=0$, that is again impossible by the above discussion.
Therefore, $\Ker(L^2) \cap \Imm(L^2)=A_1\cap A_2=0$ and $\mathfrak{g}=A_1\oplus A_2$.

Since $L$ is invertible on $A_2=\Imm(L^2)$, then $A_1=\Ker(L^2)$ exhausts
the root space for $L$ with the eigenvalue $0$.

2) Clear that $\mathfrak{o} \subset A_2$.
Since $L(U)=\alpha \cdot U-\beta \cdot V$ and $L(V)=\beta \cdot U+\alpha \cdot V$, then
$$
L^2(U)=(\alpha^2-\beta^2)\cdot U-2\alpha\beta \cdot V, \quad L^2(V)=2\alpha\beta \cdot U + (\alpha^2-\beta^2)\cdot V.
$$
By (\ref{eq.nilport1}), we get
$$
g(L^2(U),L(U))=0, \quad g(L^2(V),L(V))=0, \quad g(L^2(U),L(V))+g(L^2(V),L(U))=0
$$  on $M$. These three equalities could be re-written as follows:

$$
\left(%
\begin{array}{rrr}
 \alpha(\alpha^2-\beta^2) & -\beta(3\alpha^2-\beta^2) & 2\alpha \beta^2\\
 2\alpha \beta^2 & \beta(3\alpha^2-\beta^2) & \alpha(\alpha^2-\beta^2)\\
 \beta(3\alpha^2-\beta^2) & 2\alpha(\alpha^2-3\beta^2) &  -\beta(3\alpha^2-\beta^2) \\
\end{array}%
\right)
\left(%
\begin{array}{c}
g(U,U)\\
g(U,V)\\
g(V,V)\\
\end{array}%
\right)=\left(%
\begin{array}{c}
0\\
0\\
0\\
\end{array}\right).
$$

The equality $g(U,U)=0$ is impossible, since $U$ is non-trivial.
Hence, we have nontrivial solution $(g(U,U),g(U,V),g(V,V))$ of the above homogeneous linear system with the
determinant $-2\alpha(\alpha^2+\beta^2)^4$. Since
$\alpha^2+\beta^2\neq 0$, we get $\alpha=0$.
Moreover, $\alpha=0$ and $\beta \neq 0$ imply obviously
$g(U,V)=0$, and $g(U,U)=g(V,V)$ on $M$.

3) Recall that $L$ has no non-trivial real eigenvalue by 1) in Proposition \ref{the.eigen0}.
This observation and 2) imply that all eigenvalues of $L$ have trivial real parts.
\end{proof}

\smallskip

{\it In what follows,
we use the notation $A_1=\Ker(L^2)$ and $A_2=\Imm(L^2)$ as in Theorem \ref{the.eigen1}.}

\smallskip

\begin{remark}
Note that $\mathfrak{g}=A_1\oplus A_2=\Ker(L^2)\oplus \Imm(L^2)$ is the Fitting decomposition (see e.~g. Lemma 5.3.11 in \cite{HilNeeb})
for the operator $L$.
It should be noted that the decomposition $\mathfrak{g}=\Ker(L)\oplus \Imm(L)$ is not valid
at least for $X\in C(\mathfrak{n}(\mathfrak{g}))\setminus C(\mathfrak{g})$
(see Theorem \ref{on_nilrad1}) since
$\Imm(L)\subset C(\mathfrak{n}(\mathfrak{g})) \subset \Ker(L)$ in this case.
\end{remark}

\begin{remark}
By 2) of Proposition \ref{the.eigen0} we have $[L(Y),Y]=[[X,Y],Y]\not \in L(\mathfrak{g})$ for any $Y\in A_1 \setminus \Ker(L)$.
On the other hand, $[[X,Y],Y]\in \Ker(L)$. Indeed, $L^2(A_1)=0$ and
$$
L([[X,Y],Y])=[X,[[X,Y],Y]]=[[X,[X,Y]],Y]+[[X,Y],[X,Y]]=0.
$$
\end{remark}

\begin{remark} It is interesting to study KVFCL $X$ with $\Ker(L) \neq A_1$. For such $X$ the operator $L=\ad(X)$ is not semisimple.
One class of suitable examples are $X\in C(\mathfrak{n}(\mathfrak{g}))$ for geodesic orbit spaces $(G/H, g)$ as in Theorem
\ref{on_nilrad1} (if there is a vector $V\in \mathfrak{g}\setminus \mathfrak{n}(\mathfrak{g})$ such that $[X,V]\neq 0$).
It is not clear whether there is such KVFCL $X$ on a homogeneous Riemannian space $(G/H, g)$ with semisimple~$G$.
\end{remark}

 We will use the Jordan decomposition $L=L_s+L_n$, where $L_s$ and $L_n$ are semisimple and nilpotent part of $L=\ad(X)$ respectively.
Note that $L_s$ and $L_n$ are derivations on $\mathfrak{g}$ that are vanished on $\Ker(L)$, see e.~g. Propositions 5.3.7 and 5.3.9 in \cite{HilNeeb}.
\smallskip

Let us consider all complex conjugate pairs of eigenvalues $\pm   \beta_j\, {\bf i\,}$, $0<\beta_1<\beta_2<\cdots <\beta_l$, for $L$ (${\bf i\,}=\sqrt{-1}$).
Note that a semisimple part $L_s$ of the operator
$L$ has the same eigenvalues.
Let us consider
\begin{equation}\label{eq.rootsp}
V_j=\left\{Y\in A_2\,|\, \Bigl(L^2+\beta^2_j \Id\Bigr)^m (Y)=0 \mbox{ for some } m \in \mathbb{N}\right\},
\end{equation}
the root space of $L$ for the
pair $\pm \beta_j\,{\bf i\,}  $ (equivalently, the root space of $L^2$ for the
eigenvalue $-\beta^2_j$), $1\leq j \leq l$. It is easy to see that
$V_j=\{Y\in A_2\,|\, L_s^2 (Y)=-\beta^2_j Y\}$.

We can get a more detailed information (compare  with \cite[Section 4]{Nik2015}).
Let us consider a linear operator
\begin{equation}\label{eq.sigma1}
\sigma:A_2\rightarrow A_2,\quad \sigma(U)=\frac{1}{\beta_j}\, L_s (U), \quad U\in V_j.
\end{equation}
It is clear that $\sigma^2 =-\Id$, hence $\sigma$ plays the role of a complex structure. In what follows, we put $V_0:=\Ker(L^2)=\Ker(L_s)=A_1$.
For every $Y,Z\in A_2$ we consider

\begin{equation}\label{eq.sigma2}
[Y,Z]^+=\frac{1}{2}\bigl([Y,Z]+[\sigma(Y),\sigma(Z)] \bigr), \quad [Y,Z]^-=\frac{1}{2}\bigl([Y,Z]-[\sigma(Y),\sigma(Z)] \bigr).
\end{equation}

It is clear that $[Y,Z]=[Y,Z]^+ +[Y,Z]^-$ and, moreover, the following result holds.

\begin{prop}\label{the.eigen2}
For every $Y \in V_i$ and $Z\in V_j$, $i,j \geq 1$, we have $[Y,Z]^+ \in V_k$, $[Y,Z]^- \in V_l$, where
$\beta_k=|\beta_i-\beta_j|$ and $\beta_l=\beta_i+\beta_j$
{\rm(}if $-\beta_k^2$ {\rm(}$-\beta_l^2${\rm)} is not an eigenvalue of $L^2$, then $[Y,Z]^+=0$ {\rm(}respectively, $[Y,Z]^-=0${\rm))}.
In particular, $[V_i,V_j]\subset V_p\oplus V_q$, where $V_p$ {\rm(}$V_q${\rm)} are supposed to be trivial if
$-\beta_p^2$ {\rm(}respectively, $-\beta_q^2${\rm)} is not an
eigenvalue of $L^2$.
Moreover, $[V_0,V_i]\subset V_i$ and $[V_i,V_j] \subset A_2$ for $i \neq j$. In particular, $[A_1,A_2]\subset A_2$.
\end{prop}

\begin{proof} Straightforward calculations using (\ref{eq.sigma1}) and (\ref{eq.sigma2}) imply that
$$
L_s^2([Y,Z]^+)=-(\beta_i-\beta_j)^2\cdot [Y,Z]^+, \quad L_s^2([Y,Z]^-)=-(\beta_i+\beta_j)^2 \cdot [Y,Z]^-.
$$
If $U\in V_0=A_1$ and $Y\in V_i$, then $L_s([U,Y])=[L_s(U),Y]+[U,L_s(Y)]=[U,L_s(Y)]$ and
$L_s^2([U,Y])=[U,L_s^2(Y)]=-\beta_i^2\cdot [U,Y]$, hence, $[V_0,V_i]\subset V_i$.
These arguments prove the proposition.
\end{proof}

\begin{prop}\label{the.eigen3}
In the above notations and assumptions, the following assertions hold.

1) $C(\mathfrak{n}(\mathfrak{g}))\subset \Ker(L)\subset A_1$.

2) If $J=\{Y\in \mathfrak{g} \,|\, [C(\mathfrak{n}(\mathfrak{g})), Y]=0\}$, then $J$ is an ideal in $\mathfrak{g}$,
such that $X\in J$ and $A_2 \subset L(\mathfrak{g})=[X,\mathfrak{g}]\subset J$.

3) If $I=\{Y\in A_1 \,|\, g(Y,A_2)=0 \mbox{ on }M\}$, then $I$ is an ideal in $A_1$, such that $X\in I$, $L(A_1)=[X,A_1]\subset I$, and
$[L(A_1),A_1]=[[X,A_1],A_1]\subset I$.

4) $\widetilde{J}:=A_2+[A_2,A_2]$ is an ideal in $\mathfrak{g}$, $\widetilde{J} \subset J$, and
$g(I\cap C(\mathfrak{n}(\mathfrak{g})),\widetilde{J})=0$ on $M$.

5) $L(\widetilde{J} \cap A_1)\subset \widetilde{J} \cap I$ and  $L_s(\widetilde{J} \cap A_1)=0$.
\end{prop}

\begin{proof} 1)
We get $C(\mathfrak{n}(\mathfrak{g}))\subset \Ker(L) \subset A_1$ by 4) of Proposition \ref{the.eigen0}.

2) If $[C(\mathfrak{n}(\mathfrak{g})), Y]=0$, then for any $Z\in \mathfrak{g}$ we have (recall that $C(\mathfrak{n}(\mathfrak{g}))$ is an ideal in $\mathfrak{g}$)
$$
[C(\mathfrak{n}(\mathfrak{g})), [Z, Y]]\subset [[C(\mathfrak{n}(\mathfrak{g})),Z], Y]+[Z,[C(\mathfrak{n}(\mathfrak{g})), Y]]\subset
[C(\mathfrak{n}(\mathfrak{g})), Y]=0,
$$
hence, $J$ is an ideal in $\mathfrak{g}$.
We know that $X\in J$ by 1), hence, $A_2 \subset L(\mathfrak{g})=[X, \mathfrak{g}] \subset J$.

3) Since $g(I,A_2)=0$ on $G/H$ and $[A_1,A_2]\subset A_2$, we get $0=A_1\cdot g(I,A_2)=g([A_1,I],A_2)$, therefore, $[A_1,I]\subset I$.
We know that $g(X,A_2)=0$ due to $A_2\subset [X,\mathfrak{g}]$ and Proposition \ref{prop.1n}, hence $X\in I$. Therefore, $L(A_1)=[X,A_1]\subset I$ and
$[L(A_1),A_1]=[[X,A_1],A_1]\subset I$.

4) Since $[A_1,A_2]\subset A_2$, $\widetilde{J}=A_2+[A_2,A_2]$ is an ideal in $\mathfrak{g}$.
Since $[C(\mathfrak{n}(\mathfrak{g})),A_2]=0$, we get $\widetilde{J} \subset J$. Since
$g(I\cap C(\mathfrak{n}(\mathfrak{g})), A_2)=0$ and $\widetilde{J}$  is generated by $A_2$, we get
$$
g(I\cap C(\mathfrak{n}(\mathfrak{g})), [A_2,A_2])=A_2\cdot g(I\cap C(\mathfrak{n}(\mathfrak{g})), A_2)=0
$$
and $g(I\cap C(\mathfrak{n}(\mathfrak{g})), \widetilde{J})=0$.

5) $L(\widetilde{J} \cap A_1)\subset L(A_1)\subset I$ by 3) and $L(\widetilde{J} \cap A_1)\subset [X, \widetilde{J}] \subset \widetilde{J}$
since $\widetilde{J}$ is an ideal in $\mathfrak{g}$.
Note that $L_s(\widetilde{J} \cap A_1)=0$ follows from $A_1=\Ker(L_s)$.
\end{proof}
\medskip

\begin{remark}
Since $[A_1,A_2]\subset A_2$, then for any ideal $\mathfrak{i}$ of $\mathfrak{g}$ with the property $\mathfrak{i}\subset A_1$ we have $[\mathfrak{i},A_2]=0$.
\end{remark}

\begin{prop}\label{the.eigen4} Suppose that $X\in \mathfrak{g}$ has constant length on $(M,g)$. Then the following assertions hold.

1) If  $Y,Z\in \mathfrak{g}$ are such that
$[[X,Y],Z]\in [X,\mathfrak{g}]$, then $g([X,Y],[X,Z])=0$ on $M$.

2) If $V_i$ and $V_j$ are the root spaces (\ref{eq.rootsp}) with $i \neq j$, then
$g(V_i,V_j)=0$ on $M$.

3) If $i \neq j$ and $\beta_j\neq 2\beta_i$, then $g([V_i,V_j],V_i)=0$ and $g([V_i,V_i],V_j)=0$ on $M$.
\end{prop}

\begin{proof} 1) If $U\in \mathfrak{g}$ is such that $[[X,Y],Z]=[X,U]$, then
$$
g([Z,[Y,X]],X)=g([[X,Y],Z],X)=g([X,U],X)=0.
$$
Then we  get
$g([Y,X],[Z,X])=g([Z,[Y,X]],X)+g([Y,X],[Z,X])=0$ by (\ref{e00}).

2) For every  $Y\in V_i$ and  $Z \in V_j$ we get $[X,Y] \in V_i$ and $[[X,Y],Z]\in A_2 \subset [X,\mathfrak{g}]$ by Proposition \ref{the.eigen2}. Therefore,
$g([X,Y],[X,Z])=0$ on $M$ by 1). Since $L=\ad(X)$ is invertible on~$A_2$, we get $g(V_i,V_j)=0$ on $M$.

3) If $\beta_j\neq 2\beta_i$, then $[V_i,V_j]\subset \oplus_{k\neq i} V_k$ by Proposition \ref{the.eigen2}. Hence, $g([V_i,V_j],V_i)=0$ by 2). Therefore,
$g([V_i,V_i],V_j)=g([V_i,V_i],V_j)+g(V_i,[V_i,V_j])=V_i \cdot g(V_i,V_j)=0$.
\end{proof}

Since always $X\in \Ker(L)=\Ker(\ad(X))\subset A_1$, we get $A_1 \neq 0$. On the other hand, it is possible that $A_1=\mathfrak{g}$ and $A_2=0$
(see Remark \ref{rem7} and Proposition \ref{goabdev}).

\begin{prop}\label{the.eigen4.5} Suppose that $X\in \mathfrak{g}$ has constant length on $(M,g)$ and $A_1=\mathfrak{g}$.
Then $X \in \mathfrak{n(g)}$.
\end{prop}

\begin{proof}
If $A_1=\mathfrak{g}$, then $L^2=(\ad(X))^2=0$ on $\mathfrak{g}$.
Elements $X\in \mathfrak{g}$ with this property are called {\it absolute zero divisors} in $\mathfrak{g}$.
Using the Levi decomposition $\mathfrak{g}=\mathfrak{r(g)}\rtimes \mathfrak{s}$,
one can show that $X\in \mathfrak{r(g)}$.
If $X=X_{\mathfrak{r(g)}}+X_{\mathfrak{s}}$, where $X_{\mathfrak{r(g)}}\in \mathfrak{r(g)}$ and $X_{\mathfrak{s}}\in \mathfrak{s}$,
then for any $Y\in\mathfrak{s}$ we have
$(\ad(X))^2(Y)=[X_{\mathfrak{s}},[X_{\mathfrak{s}},Y]]+Z$,
where $Z=[X,[X_{\mathfrak{r(g)}},Y]]+[X_{\mathfrak{r(g)}},[X_{\mathfrak{s}},Y]]\in \mathfrak{r(g)}$.
Hence, $(\ad(X))^2=0$ imply $[X_{\mathfrak{s}},[X_{\mathfrak{s}},Y]]=0$ for all $Y\in\mathfrak{s}$, i.~e. $X_{\mathfrak{s}}$  is
an absolute zero divisor in $\mathfrak{s}$, that impossible for $X_{\mathfrak{s}}\neq 0$.
Indeed, if $U \in \mathfrak{s}$ is a non-trivial absolute zero divisor, then $U$ is a non-trivial nilpotent element in $\mathfrak{s}$.
Hence, there are $V,W \in \mathfrak{s}$
such that $[W,U]=2U$, $[W,V]=-2V$, and $[U,V]=W$ by the Morozov--Jacobson theorem (see e.~g. Theorem 3 in \cite{Jacob}).
But this implies $[U,[U,V]]=-2U$ that contradicts to $(\ad(U))^2=0$ (see \cite{Kost} for a more detailed discussion).
Therefore, we get $X_{\mathfrak{s}}=0$ and $X\in \mathfrak{r(g)}$.

Moreover, it is known that $\mathfrak{n(g)}=\{Y \in  \mathfrak{r(g)}\,|\, (\ad(Y))^p=0 \mbox{ for some } p\in \mathbb{N}\}$,
see e.~g. Remark 7.4.7 in \cite{HilNeeb}. Therefore, $X\in  \mathfrak{n(g)}$.
\end{proof}
\bigskip

\section{KVFCL on geodesic orbit spaces}

Let $(M=G/H, g)$ be a homogeneous Riemannian space, where $H$ is a compact subgroup
of a Lie group $G$ and $g$ is a $G$-invariant Riemannian metric.
We will suppose that $G$ acts effectively on $G/H$ (otherwise it is possible to
factorize by $U$, the maximal normal subgroup of $G$ in $H$).

We recall the definition of one important  subclass of homogeneous Riemannian spaces.

A Riemannian manifold $(M,g)$ is called {\it a  manifold with
homogeneous geodesics or a geodesic orbit manifold} (shortly,  {\it GO-manifold}) if any
geodesic $\gamma $ of $M$ is an orbit of a 1-parameter subgroup of
the full isometry group of $(M,g)$. Obviously, any geodesic orbit manifold is homogeneous.
A Riemannian homogeneous space $(M=G/H,g)$
is called {\it a space with homogeneous geodesics or a geodesic orbit space}
(shortly,  {\it GO-space}) if any geodesic $\gamma $ of $M$ is an orbit of a
1-parameter subgroup of the group $G$.
Hence, a Riemannian manifold $(M,g)$ is  a geodesic orbit Riemannian manifold,
if it is a geodesic orbit space with respect to its full connected isometry group. This terminology was introduced in~\cite{KV}
by O.~Kowalski and L.~Vanhecke, who initiated a systematic study of such spaces.
In the same paper, O.~Kowalski and L.~Vanhecke classified all GO-spaces
of dimension $\leq 6$.
A detailed exposition on geodesic orbit spaces and some important subclasses could be found also in  \cite{Arv2017,Du2018,Nik2016},
see also the references therein.
In particular, one can find many interesting results  about  GO-manifolds
and its subclasses in \cite{AA, AN, AV, BerNikClif, BerNik2012, BerNik2018, CN2017, Gor96, GorNik2018, NikNik2018, Storm2018, Tam, Ta, Yakimova, Zil96}.

Recall that all symmetric, weakly symmetric, normal homogeneous, naturally reductive, generalized normal homogeneous, and Clifford--Wolf homogeneous
Riemannian spaces are geodesic orbit, see \cite{CN2017}.
Besides the above examples, every isotropy irreducible Riemannian space is naturally reductive, and hence geodesic orbit, see e.~g.~\cite{Bes}.

\medskip

The following simple result is very useful ($M_x$ denotes the tangent space to $M$ at the point $x\in M$).

\begin{lemma}[\cite{Nik2013n}, Lemma 5]\label{newcrit}
Let $(M,g)$ be a Riemannian manifold and $\mathfrak{g}$
be its Lie algebra of Killing vector fields.
Then $(M,g)$
is a GO-manifold if and only if for any $x\in M$ and any $v\in M_x$ there is
$X\in \mathfrak{g}$ such that $X(x)=v$ and $x$ is a critical point of the function
$y\in M\mapsto g_y(X,X)$. If $(M,g)$ is homogeneous, then the latter condition is equivalent
to the following one: for any $Y\in \mathfrak{g}$ the equality $g_x([Y,X],X)=0$ holds.
\end{lemma}

Now, we recall the following remarkable result.

\begin{prop}[\cite{Nik2013n}, Theorem 1]\label{ideal}
Let $(M,g)$ be a GO-manifold, $\mathfrak{g}$ is its Lie algebra of Killing vector fields.
Suppose that $\mathfrak{a}$ is an abelian ideal of $\mathfrak{g}$. Then any $X\in \mathfrak{a}$
has constant length on $(M,g)$.
\end{prop}

As is noted in \cite{Nik2013n}, Proposition \ref{ideal} could be generalized for geodesic orbit spaces.
For the reader's convenience, we provide also the proof of the corresponding result.

\begin{theorem}\label{ideal2}
Let $(M=G/H,g)$ be a geodesic orbit Riemannian space.
Suppose that $\mathfrak{a}$ is an abelian ideal of $\mathfrak{g}=\operatorname{Lie}(G)$. Then any $X\in \mathfrak{a}$ (as a Killing vector field)
has constant length on $(M,g)$. As a corollary, $g(X,Y)\equiv \const$ on $M$ for every $X,Y \in \mathfrak{a}$.
\end{theorem}

\begin{proof}
Let $x$ be any point in $M$. We will prove that $x$ is a critical
point of the function $y\in M\mapsto g_y(X,X)$. Since $(M=G/H,g)$ is
homogeneous, then (by Lemma \ref{newcrit}) it suffices to prove
that $g_x([Y,X],X)=0$ for every $Y \in \mathfrak{g}$.

Consider any $Y\in \mathfrak{a}$, then $Y\cdot g(X,X)=2g([Y,X],X)=0$ on $M$, since $\mathfrak{a}$ is abelian.

Now, consider $Y\in \mathfrak{g}$ such that $g_x(Y,U)=0$ for every
$U \in \mathfrak{a}$. We will prove that $g_x([Y,X],X)=0$. By
Lemma \ref{newcrit}, for the vector $X(x)\in M_x$ there is a
Killing field $Z\in \mathfrak{g}$ such that $Z(x)=X(x)$ and
$g_x([V,Z],Z)=0$ for any $V\in \mathfrak{g}$. In particular,
$g_x([Y,Z],Z)=0$. Now, $W=X-Z$ vanishes at $x$ and we get
$$
g_x([Y,X],X)=g_x([Y,Z+W],Z+W)=g_x([Y,Z+W],Z)=
$$
$$
g_x([Y,Z],Z)+g_x([Y,W],Z)=g_x([Y,W],Z).
$$
Note that $g_x([Y,W],Z)=-g_x([W,Y],Z)=g_x(Y,[W,Z])=0$ due to
$W(x)=0$ ($0=W\cdot g(Y,Z)|_x=g_x([W,Y],Z)+g_x(Y,[W,Z])$) and
$[W,Z]=[X,Z] \in \mathfrak{a}$. Therefore, $g_x([Y,X],X)=0$.

Since every $Y\in \mathfrak{g}$ could be represented as $Y=Y_1+Y_2$, where $Y_1\in \mathfrak{a}$ and $g_x(Y_2,\mathfrak{a})=0$, then
$x$ is a critical point of the function $y\in M\mapsto g_y(X,X)$.
Since every $x\in M$ is a critical
point of the function $y\in M\mapsto g_y(X,X)$, then
$X$ has constant length on $(M,g)$.

The last assertion follows from the equality $2g(X,Y)=g(X+Y,X+Y)-g(X,X)-g(Y,Y)$.
\end{proof}

\begin{corollary}\label{ideal3}
Every geodesic orbit Riemannian space $(M=G/H, g)$ with non-semisimple group $G$ has non-trivial
Killing vector fields of constant length.
\end{corollary}

\begin{proof} If the Lie algebra $\mathfrak{g}=\operatorname{Lie}(G)$ is non semisimple, then it has a non-trivial abelian ideal $\mathfrak{a}$
(for instance, this property has the center of the nilradical $\mathfrak{n(g)}$ of $\mathfrak{g}$).
Now, it suffices to apply Theorem \ref{ideal2}.
\end{proof}

\smallskip

We recall some other important properties of geodesic orbit spaces.
Any semisimple Lie algebra $\mathfrak{s}$ is a direct Lie algebra sum of its compact and noncompact parts
($\mathfrak{s}=\mathfrak{s}_{c} \oplus \mathfrak{s}_{nc}$).
The following proposition is asserted in \cite{Gor96}, a detailed proof could be found in \cite{GorNik2018}.

\begin{prop}\label{gnc}
Let $(G/H, g)$ be a connected geodesic orbit space and let $\mathfrak{s}$ be any Levi factor of $G$.
Then the noncompact part $\mathfrak{s}_{nc}$ of $\mathfrak{s}$ commutes with the radical $\mathfrak{r}(\mathfrak{g})$.
\end{prop}

\begin{remark}\label{dirdecgo}
For a geodesic orbit space $(G/H, g)$ we have a direct Lie algebra sum \linebreak
$\mathfrak{g}=(\mathfrak{r(g)}\rtimes \mathfrak{s}_c)\oplus \mathfrak{s}_{nc}$
by Proposition \ref{gnc}. Moreover, we can represent $\mathfrak{s}_{nc}$ as a direct sum of simple noncompact ideals.
This decomposition is useful for applying of Theorem \ref{decomp1}.
\end{remark}

\begin{prop}[C.~Gordon, \cite{Gor96}]\label{strucnilr}
Let $(G/H, \rho)$ be a geodesic orbit space. Then the nilradical $\mathfrak{n(g)}$ of the Lie algebra $\mathfrak{g}=\Lie(G)$
is commutative or two-step nilpotent.
\end{prop}

Suppose that  $X\in \mathfrak{g}$ has constant length on a geodesic orbit space $(G/H, g)$, then
$[A_1,A_2]\subset A_2$ and
$C(\mathfrak{n}(\mathfrak{g}))\subset A_1$ by Propositions \ref{the.eigen2} and  \ref{the.eigen3} for a given Killing field  $X$ of constant length
($A_1=\Ker(L^2)$ and $A_2=\Imm(L^2)$ as in Theorem \ref{the.eigen1}).
Moreover,  $\mathfrak{n}(\mathfrak{g})$ is at most 2-step nilpotent by Proposition \ref{strucnilr}. All these considerations lead to the following

\begin{conjecture}\label{conjec3}
If $X\in \mathfrak{g}$ has constant length on a geodesic orbit space $(G/H, g)$, then
$\mathfrak{n}(\mathfrak{g}) \subset A_1$.
\end{conjecture}

\begin{theorem}\label{on_nilrad1}
For a geodesic orbit space $(G/H, g)$, we consider any  $X \in \mathfrak{n}(\mathfrak{g})$.
Then the following conditions are equivalent:

1) $X$ is in the center $C(\mathfrak{n}(\mathfrak{g}))$ of $\mathfrak{n}(\mathfrak{g})$;

2) $X$ has constant length on $(G/H, g)$.
\end{theorem}

\begin{proof} $1) \Rightarrow 2)$. Since the center $C(\mathfrak{n}(\mathfrak{g}))$ is an abelian ideal in $\mathfrak{g}$,
then $X$ has constant length due to Theorem \ref{ideal2}.

$2) \Rightarrow 1)$. Since $X\in \mathfrak{n}(\mathfrak{g})$ and $\mathfrak{n}(\mathfrak{g})$ is at most two step nilpotent by Proposition \ref{strucnilr},
then $[Z,[Z,X]]=0$ for any $Z \in \mathfrak{n}(\mathfrak{g})$. Now, by
Lemma \ref{le1}, we have
$$
g([Z,X],[Z,X])=g([Z,[Z,X]],X)+g([Z,X],[Z,X])=0,
$$
hence $[Z,X]=0$ for any $Z \in \mathfrak{n}(\mathfrak{g})$. Consequently, $X\in C(\mathfrak{n}(\mathfrak{g}))$.
\end{proof}

\begin{corollary}
Under conditions of Theorem \ref{on_nilrad1}, any abelian ideal $\mathfrak{a}$ in $\mathfrak{g}$ is in $C(\mathfrak{n}(\mathfrak{g}))$. In particular,
$C(\mathfrak{n}(\mathfrak{g}))$ is a maximal abelian ideal  in $\mathfrak{g}$ by inclusion.
\end{corollary}

\begin{proof}
It is clear that $\mathfrak{a}$ is a nilpotent ideal in $\mathfrak{g}$, hence $\mathfrak{a}\subset \mathfrak{n}(\mathfrak{g})$.
By Proposition \ref{strucnilr}, $\mathfrak{a}$ consists of Killing fields of constant length,
hence, $\mathfrak{a}\subset C(\mathfrak{n}(\mathfrak{g}))$ by Theorem \ref{on_nilrad1}.
\end{proof}

\bigskip

\begin{remark}\label{rem7}
It should be recalled that there are many examples of geodesic orbit nilmanifolds~\cite{Gor96}.
Therefore, Theorems \ref{ideal2} and \ref{on_nilrad1} give non-trivial examples $X$ of KVFCL on  $(M=G/H,g)$, where
$X \in C(\mathfrak{n}(\mathfrak{g}))$.
For any such example, the operator $\ad(X):\mathfrak{g}\rightarrow \mathfrak{g}$ is non semisimple, since it is nilpotent.
In this case, $A_1=\mathfrak{g}$ and $L^2=(\ad(X))^2=0$.
For semisimple $\mathfrak{g}$, there is no counterexample for Conjecture  \ref{conjec2}.
\end{remark}

Let us recall Problem~2 in \cite{Nik2015}: {\it Classify geodesic orbit Riemannian spaces with nontrivial Killing vector fields
of constant length}. Now, this problem is far from being resolved. We have one modest result in this direction.

\begin{prop}\label{goabdev}
Let $(G/H, g)$ be a geodesic orbit space and $X\in \mathfrak{g}=\operatorname{Lie}(G)$. Then the following conditions are equivalent:

1) $X$ has constant length on $(G/H, g)$ and $A_1=\Ker(L^2)=\mathfrak{g}$;

2) $X$ is in the center $C(\mathfrak{n}(\mathfrak{g}))$ of $\mathfrak{n}(\mathfrak{g})$.
\end{prop}

\begin{proof} $1) \Rightarrow 2)$.
By Proposition \ref{the.eigen4.5}, we get $X \in \mathfrak{n(g)}$.
Hence, $X \in C(\mathfrak{n}(\mathfrak{g}))$ by Theorem~\ref{on_nilrad1}.

$2) \Rightarrow 1)$.
By Theorem~\ref{on_nilrad1}, $X$ has constant length on $(G/H, g)$.
Since $\mathfrak{n}(\mathfrak{g})$ is an ideal in~$\mathfrak{g}$, then $[X,Y]\in \mathfrak{n}(\mathfrak{g})$ and
$L^2(Y)=[X,[X,Y]]\in [C(\mathfrak{n}(\mathfrak{g})),\mathfrak{n}(\mathfrak{g})]=0$ for all $Y \in \mathfrak{g}$.
\end{proof}

\vspace{3mm}

{\bf Acknowledgements.}
The author is indebted to Prof. V.N.~Berestovskii for helpful discussions concerning this paper.

\bibliographystyle{amsunsrt}

\vspace{10mm}

\end{document}